\documentclass[a4paper,11pt]{article}
\title{A characterization of pseudo complete finitely ramified valued fields through a Hahn-like construction}
\author{Anna De Mase\protect{\footnote{Dipartimento di Matematica e Fisica, Università degli Studi della Campania ``Luigi Vanvitelli", viale Lincoln 5, Caserta (Italy), anna.demase@unicampania.it}}}
\date{}
\usepackage{amsmath}
\usepackage{amsthm}
\usepackage{wasysym}
\usepackage{amssymb}
\usepackage{amsfonts}
\usepackage{tikz}
\usepackage[hyperfootnotes]{hyperref}
\usepackage{times}
\usetikzlibrary{matrix}
\begin{document}
\maketitle
\newtheorem{thm} {Theorem}[section]
\newtheorem{defn} [thm] {Definition}
\newtheorem{cor} [thm] {Corollary}
\newtheorem{rem} [thm] {Remark}
\newtheorem{lem} [thm] {Lemma}
\newtheorem{prop} [thm] {Proposition}
\newtheorem{fact} [thm] {Fact}
\newtheorem{conj} [thm] {Conjecture}
\begin{abstract} We give a characterization of finitely ramified $\omega$-pseudo complete valued fields of mixed characteristic $(0,p)$, with fixed residue field $k$ and value group $G$ of cardinality $\aleph_{1}$ in terms of a Hahn-like construction over the Cohen field $C(k)$, modulo the Continuum Hypothesis. This is a generalization of results due to Ax and Kochen in '65 for formally $p$-adic fields and by Kochen in '74 for unramified valued fields with perfect residue field. We consider a more general context of finitely ramified valued fields of mixed characteristic with arbitrary residue field and a cross-section. \end{abstract}
\vspace{2mm}
\section*{Introduction}
In their paper (\cite{AxK65}), Ax and Kochen give a complete axiomatization of the theory of the class of formally $p$-adic fields, by the following properties: \begin{itemize} \item the value group is a $\mathbb{Z}$-group; \item the residue field is $\mathbb{F}_{p}$; \item $v(p)$ is the minimal positive element of the value group. \end{itemize} Thus, a valued field $K$ is formally $p$-adic if and only if it is elementarily equivalent to $\mathbb{Q}_{p}$.
In \cite{20padic} Macintyre gives a broad analysis of the field $\mathbb{Q}_{p}$, both from an algebraic and a model theoretical point of view, focusing also on its comparison with the field of real numbers. Indeed, they share the property of being completions of the field of rational numbers, $\mathbb{R}$ with respect to the usual archimedean absolute value, and $\mathbb{Q}_{p}$ with respect to the non-archimedean one associated to the $p$-adic valuation. Hasse's local-global principle (\cite{HasseH24}) shows the role played by both fields in the existence of solutions of rational equations. We are mainly interested in a property not shared by both fields. It is a well-known property of real closed fields that a field elementarily equivalent to a finite extension of the field of real numbers $\mathbb{R}$ contains a subfield that is elementarily equivalent to $\mathbb{R}$. This is not the case for formally $p$-adic fields. Indeed, in \cite{BvDDM88} the authors show that for any finite extension $L$ of $\mathbb{Q}_{p}$, there is a valued field elementarily equivalent to $L$ that does not have any subfield elementarily equivalent to $\mathbb{Q}_{p}$. This ``pathology" of the $p$-adics can be avoided considering $\omega$-pseudo complete formally $p$-adic fields, which have a subfield that is an isomorphic copy of the field of $p$-adic numbers, as it follows as a particular case of Lemma \ref{C}. There we show that all henselian valued fields with a fixed residue field $k$ and discrete value group $G$ have $C(k)$, the Cohen ring over $k$, as a common valued subfield (for an introduction to Cohen rings see \cite{cohen46}, and for a recent model-theoretic analysis see \cite{AJ19cohen}). In particular, we note that if $K$ has finite ramification $e$, then $K$ contains a copy of a finite extension of $C(k)$, that is the totally ramified extension of $C(k)$ of ramification $e$. This is obtained by adding to $C(k)$ a root of an Eisenstein polynomial. We then use this property to generalize a result due to Ax and Kochen in \cite{AxK65}, which gives a characterization of $\omega$-pseudo complete valued fields elementarily equivalent to the field of $p$-adic numbers $\mathbb{Q}_{p}$. In particular, in Theorems 1 and 2 of \cite{AxK65}, assuming the Continuum Hypothesis, the authors give a concrete description of $\omega$-pseudo complete formally $p$-adic fields with a fixed value group $G$ of cardinality $\aleph_{1}$, using a Hahn-like construction over $\mathbb{Q}_{p}$ (see also \cite{Kochen74}). We will refer to this structure as the field of ``twisted" power series since it is characterized by having the multiplication defined by adding an extra factor that is a power of $p$, and whose exponent is given by a $2$-cocycle defined through the value group. The addition is defined coefficientwise as the usual addition of power series. If $G$ is a $\mathbb{Z}$-group, this Hahn-like construction over $\mathbb{Q}_{p}$ gives a valued field of mixed characteristic that is a formally $p$-adic field. In Section \ref{twisted}, we extend the construction to a more general setting, 
i.e. we define the field of ``twisted" power series over any mixed characteristic valued field which has a cross-section (recall that a cross-section for a valuation $v: K\longrightarrow G$ is a group homomorphism $\sigma: G\longrightarrow K^{\times}$ such that $v(\sigma(g))=g$ for all $g\in G^{\times}$). 

The problem of characterizing valued fields in terms of power series has been addressed by many mathematicians, in both the equicharacteristic case (by Ax and Kochen in \cite{AxK65i} for the equicharacteristic $0$ case, and by Kaplansky, Schilling, and Teichm\"uller in \cite{Kaplansky42}, \cite{Schilling37}, and \cite{Teich37} respectively, for the equicharacteristic $p$ case), and the $p$-adic case (by MacLane in \cite{Maclane38}, \cite{Maclane39}, and by Poonen in \cite{Poonen93}, other than by Ax and Kochen in the aforementioned paper). In particular, in the mixed characteristic case, the motivations behind the ``twisted" power series construction can be understood in the following argument. Suppose $v$ is a mixed characteristic valuation from a field $K$ to an ordered group $G$, and let $A$ be the smallest convex subgroup of $G$. In that case, one can apply the coarsening method in order to decompose $v$ into a valuation $v_{0}$ with values in $A$, and another valuation $\dot{v}$ with values in the quotient group $G/A$, which is ordered by the induced order of $G$. The existence of a cross-section $\sigma$ for the valuation $v$ does not imply that $\dot{v}$ has a cross-section, since the exact sequence \begin{equation}\label{exactsequence} 0\xrightarrow{}A\xrightarrow{i} G\xrightarrow{\rho} G/A\xrightarrow{} 0\end{equation} of the value groups does not always split. Indeed, if it splits, i.e. if there exists a homomorphism of ordered abelian groups $\alpha: G/A\longrightarrow G$ such that $\rho\circ\alpha=id_{G/A}$, then $\sigma\circ\alpha$ is a cross-section for $\dot{v}$. Furthermore, in this case the coarsening method can be inverted, in the following sense. Starting from a splitting exact sequence of ordered abelian groups as in (\ref{exactsequence}), i.e. a mixed characteristic valued field with values in $A$, and an equicharacteristic $0$ valued field with values in $G/A$, one can compose the valuations and obtain a mixed characteristic valued field with a finer valuation which takes values in the group $G$. Some examples of mixed characteristic valued fields obtained by inverting the coarsening method are $(\mathbb{Q}_{p}((t^{\mathbb{Z}})),val)$, $(\mathbb{Q}_{p}((t_{1}^{\mathbb{Z}}))\ldots((t_{n}^{\mathbb{Z}})),val_{n})$, and $(\bigcup_{n\in\mathbb{N}}\mathbb{Q}_{p}((t_{1}^{\mathbb{Z}}))\ldots((t_{n}^{\mathbb{Z}})),val_{\infty})$, which were studied by the author in \cite{ADM1}, and whose valuations are obtained by the composition of the $p$-adic and the $t$-adics valuations. Then, it is very natural to ask whether one can ``reconstruct'' a group $G$ from a fixed exact sequence of ordered abelian groups $$0\xrightarrow{}A\xrightarrow{i} G\xrightarrow{\rho} H\xrightarrow{} 0$$ that does not split, and define a mixed characteristic valued field with value group $G$. With their characterization, Ax and Kochen show that this is possible if one considers formally $p$-adic fields and an exact sequence of ordered groups as above, where $G$ is a $\mathbb{Z}$-group. This paper aims at showing an analogous characterization for any $\omega$-pseudo complete mixed characteristic henselian valued field with finite ramification, fixed residue field $k$, and value group $G$ having $\mathbb{Z}$ as the minimal convex subgroup (see Theorem \ref{mainthm} and Corollary \ref{finextscharacterization}).

\section{Exact sequences of ordered abelian groups}
We start by recalling some results on exact sequences of groups. This is done in general for modules (see e.g. \cite{Robinson96}), here we review it for abelian groups equipped with a total order.

Let $G$ be an ordered abelian group and $A$ a convex subgroup of $G$. Let $H=G/A$ be the abelian group with the order induced by the order of $G$, and consider the exact sequence   $$0\xrightarrow{}A\xrightarrow{i} G\xrightarrow{\rho} H\xrightarrow{} 0.$$ For each $h\in H$, choose $\alpha(h)\in G$ any preimage of $h$ with respect to the projection map. We call $\alpha:H\longrightarrow G$ a \textit{transversal map}. By definition of the quotient order on $H$, any choice of $\alpha$ would be order-preserving. Then $\rho\alpha=id_{H}$, and $\alpha$ does not need to be a group homomorphism. Then define the following function $\beta: H\times H\longrightarrow G$ by 
\begin{equation}\label{betadefn}\beta(h,h')=\alpha(h+h')-\alpha(h)-\alpha(h'),\end{equation}
for all $h,h'\in H$.
\begin{prop} For every $h,h',h''\in H$, $\beta$ satisfies the following property
\begin{equation}\label{betarule}\beta(h,h')+\beta(h+h',h'')=\beta(h',h'')+\beta(h,h'+h'').\end{equation}
\end{prop}
\begin{proof} By the associative law of $G$,  for every $h,h',h''\in H$, $$(\alpha(h)+\alpha(h'))+\alpha(h'')=\alpha(h)+(\alpha(h')+\alpha(h'')).$$
From (\ref{betadefn}), the following holds $$\alpha(h)+\alpha(h')=\alpha(h+h')-\beta(h,h'),$$ and so we have $$(\alpha(h+h')-\beta(h,h'))+\alpha(h'')=\alpha(h)+(\alpha(h'+h'')-\beta(h',h'')).$$ Repeating the substitution in $$\alpha(h+h')+\alpha(h'')=\alpha(h+h'+h'')-\beta(h+h',h'')$$ and $$\alpha(h)+\alpha(h'+h'')=\alpha(h+h'+h'')-\beta(h,h'+h'')$$ we have the assert. \end{proof}
In the following remark, we prove that $\beta$ takes values in $A$.
\begin{rem} Let $$0\xrightarrow{}A\xrightarrow{i} G\xrightarrow{\rho} H\xrightarrow{} 0$$ be a short exact sequence of ordered abelian groups. If $h,h'\in H$, then $\beta(h,h')\in A$. \end{rem}
\begin{proof} If $h,h'\in H$, then \begin{align*}\rho(\beta(h,h'))&=\rho(\alpha(h)+\alpha(h')-\alpha(h+h'))=\\&=\rho(\alpha(h))+\rho(\alpha(h'))-\rho(\alpha(h+h'))=\\&=h+h'-(h+h')=0,\end{align*} since $\rho\alpha=id_{H}$ and $\rho$ is an epimorphism. Thus, $$\beta(h,h')\in ker\,\rho=Im\,i\cong A.$$ \end{proof}
\begin{defn} Let $n\geq1$ be any integer. An $n$-cochain is a map $\phi: H^{n}\longrightarrow G$ such that $\phi(h_{1},\ldots,h_{n})=0$ whenever $h_{i}=0$ for some $i$. \end{defn}
\begin{defn} \begin{itemize} \item[(i)] A $2$-cochain $\beta:H\times H\longrightarrow A$ satisfying (\ref{betarule}) is called a $2$-cocycle. We denote the set of $2$-cocycles from $H$ to $A$ by $$Z^{2}(H,A).$$ \item[(ii)] A $2$-cocycle $\beta$ is called a $2$-coboundary if there exists a map $\alpha:H\longrightarrow A$ such that $$\beta(h,h)=\alpha(h+h)-\alpha(h)-\alpha(h').$$ We denote the set of the $2$-coboundaries from $H$ to $A$ by $$B^{2}(H,A).$$\end{itemize}
\end{defn}
Notice that if $\beta$ is a $2$-coboundary associated to a transversal map $\alpha$, then $\alpha$ is a $1$-cochain from $H$ to $G$. Indeed, $$0=\beta(h,0)=\alpha(h)-\alpha(h)-\alpha(0)=-\alpha(0).$$ Moreover, for every $h,h'\in H$ $$\beta(h,h')=\beta(h',h).$$
\begin{rem} The set $Z^{2}(H,A)$ with the operation defined by $$(\phi+\psi)(h,h')=\phi(h,h')+\psi(h,h')$$ for all $\phi,\psi\in Z^{2}(H,A)$ is an abelian group and $B^{2}(H,A)$ is a subgroup with the restricted operation.
\end{rem}
\begin{defn} The quotient group $$H^{2}(H,A)=Z^{2}(H,A)/B^{2}(H,A)$$ is the \textit{second cohomology group of H with coefficients in A}. \end{defn}
\begin{fact} \label{relationbetas} Let $\alpha, \alpha'$ be two transversal maps from $H$ to $G$ and $\beta, \beta'$ the $2$-cocycles such that \begin{align*}\beta(h,h')&=\alpha(h+h')-\alpha(h)-\alpha(h')\\\beta'(h,h')&=\alpha'(h+h')-\alpha'(h)-\alpha'(h')\end{align*} for every $h,h'\in H$. Then $\beta$ and $\beta'$ belong to the same coset modulo $B^{2}(H,A)$.\end{fact}
\begin{proof} Note that for every $h\in H$, $\rho(\alpha(h)-\alpha'(h))=h-h=0$, thus $\alpha(h)-\alpha'(h)$ is in $ker\rho=Im\,i\cong A$. So, there exist a map $\psi:H\longrightarrow A$, such that for every $h\in H$, $\alpha(h)-\alpha'(h)=\psi(h)$. Substituting in the definition of $\beta'$ we have \begin{align*}\beta'(h,h')&=\alpha'(h+h')-\alpha'(h)-\alpha'(h')=\\&=\alpha(h+h')-\psi(h+h')-(\alpha(h)-\psi(h))-(\alpha(h')-\psi(h'))=\\&=\beta(h,h')-(\psi(h+h')-\psi(h)-\psi(h'))=\\&=\beta(h,h')-\psi^{*}(h,h')\end{align*} where $\psi^{*}:H\times H\longrightarrow A$ is a $2$-coboundary defined by the $1$-cochain $\psi$ from $H$ to $A$.\end{proof}

\begin{lem}\label{lemmagroup} If $\beta: H\times H\longrightarrow A$ is a $2$-cocycle, the product $A\times H$ can be equipped with an ordered abelian group structure where the group operation + is defined by $$(z,h)+(z',h')=(z+z'-\beta(h,h'),h+h')$$ and the order is the lexicographic order $$(z,h)\leq(z',h')\iff h\leq h'\, or\,h=h'\, and\,z\leq z'.$$ \end{lem}
\begin{proof}We check the axioms of ordered abelian groups. Take $(z,h),(z',h')$ and $(z'',h'')$ in $A\times H$. Then
\begin{itemize}
\item \textit{Associavity.} 
\begin{align*}&((z,h)+(z',h'))+(z'',h'')=(z+z'-\beta(h,h'),h+h')+(z'',h'')=\\&=(z+z'+z''-(\beta(h,h')+\beta(h+h',h'')),h+h'+h'')=\\&=(z+z'+z''-(\beta(h',h'')+\beta(h,h'+h'')),h+h'+h'')=\\&=(z,h)+(z'+z''-\beta(h',h''),h'+h'')=(z,h)+((z',h')+(z'',h'')),\end{align*} where the forth equality holds because of (\ref{betarule}).
\item \textit{Commutativity.}
\begin{align*}&(z,h)+(z',h')=(z+z'-\beta(h,h'),h+h')=\\&=(z'+z-\beta(h',h),h'+h)=(z',h')+(z,h).\end{align*}
\item \textit{Identity.} Since $\beta(h,0)=\beta(0,h')=0$ for every $h,h'\in H$, the pair $(0,0)$ satisfies $$(z,h)+(0,0)=(z-\beta(h,0),h)=(z,h)=(z-\beta(0,h),h)=(0,0)+(z,h).$$
\item \textit{Inverse.} $$(z,h)+(-z+\beta(h,-h),-h)=(z-z+\beta(h,-h)-\beta(h,-h),h-h)=(0,0).$$
\item \textit{Compatibility with the order.} Let $(z,h)>0$. Then $(z,h)+(z',h')=(z+z'-\beta(h,h'),h+h')>(z',h')$ for every $(z',h')\in A\times H$. Indeed, either $h+h'>h'$, if $h>0$, or if $h=0$ and $z>0$, then $\beta(h,h')=\beta(0,h')=0$ and $z+z'>z'$.
\end{itemize}
\end{proof}

We denote the group structure defined on $A\times H$ by $G_{\beta}$. Thus, the sequence $$0\xrightarrow{}A\xrightarrow{i} G_{\beta}\xrightarrow{\rho} H\xrightarrow{} 0$$ is a short exact sequence of ordered abelian groups. Moreover, the following fact shows that every exact sequence of ordered abelian groups is associated to a unique equivalence class $\beta+B^{2}(H, A)$ of the cohomology group for some $\beta\in H^{2}(H,A)$.

\begin{fact} \label{isom} Let $$0\xrightarrow{}A\xrightarrow{i} G\xrightarrow{\rho} H\xrightarrow{} 0$$ be a short exact sequence of ordered abelian groups and $\alpha:H\longrightarrow G$ a transversal map defining a $2$-cocycle $\beta$. Then, $G\cong G_{\beta}$. Moreover, $G\cong G_{\beta'}$ for every $\beta'$ defined by any other transversal map $\alpha'$.
\end{fact}
\begin{proof}
Consider the map \begin{align*}
f: & \,\,\,\,G_{\beta} \,\,\,\longrightarrow \,\,\,G \\
   & \,(z,h) \,\,\mapsto z+\alpha(h)
\end{align*}
Then
\begin{itemize}
\item $f$ is a homomorphism. Indeed,
\begin{align*} f\left((z,h)+(z',h')\right)&=f(z+z'-\beta(h,h'),h+h')= \\
&=z+z'-\beta(h,h')+\alpha(h+h')=\\ 
&=z+z'+\alpha(h)+\alpha(h')= \\
&=z+\alpha(h)+z'+\alpha(h')=f(z,h)+f(z',h'). \end{align*}
\item $f$ is injective. Let $(z,h)\in A\times H$ be such that $z+\alpha(h)=0$. Then $\rho(z+\alpha(h))=\rho(z)+\rho(\alpha(h))=\rho(z)+h=0$. Thus $h=0$, since $z\in A=Im\,i=ker\rho$. From $z+\alpha(h)=0$ and $\alpha(0)=0$, it follows $z=0$.
\item $f$ is surjective. Take $g\in G$. We want $(z,h)\in A\times H$ such that $g=z+\alpha(h)$. Consider the element $g-\alpha(\rho(g))\in G$. Note that $\rho(g-\alpha(\rho(g)))=\rho(g)-\rho(g)=0$. Since $ker\rho=Im\,i=A$, then there is $z\in A$ such that $g-\alpha(\rho(g))=z$ and so $g=z+\alpha(h)$, with $h=\rho(g)$.
\item $f$ is order-preserving. Take $(z,h)\leq(z',h')$. Suppose first that $h<h'$, then $z+\alpha(h)<z'+\alpha(h')$. Indeed, suppose by contradiction that $z'+\alpha(h')\leq z+\alpha(h)$. Since $\alpha$ is order-preserving, then $0<\alpha(h')-\alpha(h)\leq z-z'$, and since $A$ is convex in $G$, then $\alpha(h')-\alpha(h)\in A$. It follows that $\rho(\alpha(h')-\alpha(h))=h'-h=0$ that contradicts $h<h'$.
Now suppose $h=h'$ and $z<z'$. Clearly $\alpha(h)=\alpha(h')$ and $z+\alpha(h)<z'+\alpha(h')$.
\end{itemize}
Now, let $\beta'$ be a $2$-cocycle defined through a different transversal map $\alpha':H\longrightarrow G$. By Fact \ref{relationbetas}, $\beta'=\beta-\psi^{*}$ where $\psi^{*}$ is a $2$-coboundary defined from a $1$-cochain $\psi:H\longrightarrow A$. Then the map \begin{align*}
g: & \,\,\,\,G_{\beta} \,\,\,\longrightarrow \,\,\,G_{\beta'} \\
   & \,(z,h) \,\,\mapsto (z+\psi(h),h)
\end{align*}
is an isomorphism of ordered abelian groups. Indeed, for every $(z,h),(z',h')\in G_{\beta}$, the following properties hold:\begin{itemize}
\item $g$ is a homomorphism. Indeed, \begin{align*} g((z,h)+(z',h'))&=g(z+z'-\beta(h,h'),h+h')=\\&=(z+z'-\beta(h,h')+\psi(h+h'),h+h')=\\&=(z+z'-\beta'(h,h')+\psi(h)+\psi(h'),h+h')=\\&=(z+\psi(h),h)+(z'+\psi(h'),h')=\\&=g(z,h)+g(z',h')\end{align*}
\item Let $(z,h)\in G_{\beta}$ such that $(z+\psi(h),h)=(0,0)$. Then $h=0$, and since $\psi$ is a $1$-cochain, $z=0$. Thus $g$ is injective.
\item Let $(\dot{z},\dot{h})\in G_{\beta'}$, then $(\dot{z}-\psi(\dot{h}),\dot{h})\in G_{\beta}$ is such that $g(\dot{z}-\psi(\dot{h}),\dot{h})=(\dot{z},\dot{h})$, which proves surjectivity of $g$.
\item Let $(z,h)\leq (z',h')$. If $h<h'$, then clearly $g(z,h)<g(z',h')$. Suppose $h=h'$ and $z\leq z'$, then $\psi(h)=\psi(h')$ and $(z+\psi(h),h)\leq(z'+\psi(h'),h')$, thus $g$ is order preserving.
\end{itemize}
\end{proof}


\section{An Hahn-like construction for mixed characteristic valued fields} \label{twisted}
In this section, we show an application of the above results in the context of valued fields. Indeed, we use the $2$-cocycle to generalize the Hahn-like construction described by Ax and Kochen (see \cite{AxK65} and \cite{Kochen74}) to finitely ramified valued fields with a cross-section. Thus, consider an exact sequence of ordered abelian groups $$0\xrightarrow{}A\xrightarrow{i} G\xrightarrow{\rho} H\xrightarrow{} 0$$ with $A$ the smallest convex subgroup of $G$ and $H=G/A$. Let $(K,v)$ be a mixed characteristic $(0,p)$ valued field valued in $A$, and such that $\sigma: A\longrightarrow K^{\times}$ is a cross-section.
Now, let $\delta$ be an infinite cardinal and consider the set $$K(t^{H},\beta)_{\delta}=\left\{\sum_{h\in S}a_{h}t^{h}\,:\, S\subseteq_{\text{w.o.}}H,\, |S|\leq\delta,\, a_{h}\in K\, \text{for  all}\,h\in S\right\}$$ where the support $S$ of an element $a=\sum_{h\in S}a_{h}t^{h}$ is a well ordered (w.o.) subset of $H$. The set $K(t^{H},\beta)_{\delta}$ can be equipped with the operations of addition and multiplication defined as follows: for all $a=\sum_{h\in S}a_{h}t^{h}$ and $b=\sum_{h\in S'}b_{h}t^{h}$ in $K(t^{H},\beta)_{\delta}$,
$$a_{h}t^{h}+b_{h}t^{h}=(a_{h}+b_{h})t^{h}$$ and $$a_{h}t^{h}b_{h'}t^{h'}=a_{h}b_{h'}\sigma(-\beta(h,h'))t^{h+h'}.$$
Note that if $A\cong\mathbb{Z}$ and $\pi$ is a uniformizer (i.e. $v(\pi)=1$), then the map $\sigma: A\longrightarrow K^{\times}$ defined as $\sigma(m)=\pi^{m}$, is a cross-section for the valuation $v$. Then $\pi^{\beta(h,h')}\in K$ for all $h,h'\in H$ and the multiplication in the above definition is well defined. 

\begin{prop}\label{field} If $\beta:H\times H\longrightarrow A$ is a 2-cocycle, the structure $$(K(t^{H},\beta)_{\delta},+,\cdot)$$ defined above is a field.\end{prop}
\begin{proof} The axioms of abelian groups hold in $(K(t^{H},\beta)_{\delta},+)$ as in the field of formal power series defined on $K$. We now check that the axioms involving multiplication also hold. Let $a=\sum_{h\in S}a_{h}t^{h},\,b=\sum_{h\in S'}b_{h}t^{h},\,c=\sum_{h\in S''}c_{h}t^{h}$ in $K(t^{H},\beta)_{\delta}$.
\begin{itemize} 
\item \textit{Associativity.}
\begin{align*}(a\cdot b)\cdot c&=\left(\sum_{h\in S}a_{h}t^{h}\cdot\sum_{h\in S'}b_{h}t^{h}\right)\cdot\sum_{h\in S''}c_{h}t^{h}= \\  &=\left[\sum_{l\in S+S'}\left(\sum_{\substack{h\in S \\ h'\in S' \\ h+h'=l}}a_{h}b_{h'}\sigma(-\beta(h,h'))\right)t^{l}\right]\cdot\sum_{h\in S''}c_{h}t^{h}=\\ &=\sum_{r\in S+S'+S''}\left(\sum_{\substack{l\in S+S' \\ h''\in S'' \\ l+h''=r}}\left(\sum_{\substack{h\in S \\ h'\in S' \\ h+h'=l}}a_{h}b_{h'}\sigma(-\beta(h,h'))\right)c_{h''}\sigma(-\beta(l,h''))\right)t^{r}= \\ &=\sum_{r\in S+S'+S''}\left(\sum_{\substack{h\in S \\ h'\in S' \\ h''\in S'' \\ h+h'+h''=r}}a_{h}b_{h'}c_{h''}\sigma(-(\beta(h,h')+\beta(h+h',h'')))\right)t^{r}
\end{align*}
Using (\ref{betarule}) in the last expression, we obtain that
\begin{align*}
(a\cdot b)\cdot c&=\sum_{r\in S+S'+S''}\left(\sum_{\substack{h\in S \\ h'\in S' \\ h''\in S'' \\ h+h'+h''=r}}a_{h}b_{h'}c_{h''}\sigma(-(\beta(h',h'')+\beta(h,h'+h'')))\right)t^{r}= \end{align*}
\begin{align*}
&=\sum_{r\in S+S'+S''}\left(\sum_{\substack{l\in S+S' \\ h\in S \\ l+h=r}}a_{h}\sigma(-\beta(h,l))\left(\sum_{\substack{h'\in S' \\ h''\in S'' \\ h'+h''=l}}b_{h'}c_{h''}\sigma(-\beta(h',h''))\right)\right)t^{r}= \\ &=\sum_{h\in S}a_{h}t^{h}\cdot\left[\sum_{l\in S'+S''}\left(\sum_{\substack{h'\in S \\ h''\in S' \\ h'+h''=l}}b_{h'}c_{h''}\sigma(-\beta(h',h''))\right)t^{l}\right]=\\ &=\sum_{h\in S}a_{h}t^{h}\cdot\left(\sum_{h\in S'}b_{h}t^{h}\cdot\sum_{h\in S''}c_{h}t^{h}\right)= a\cdot(b\cdot c).
\end{align*}
\item \textit{Commutativity.}
\begin{align*} a\cdot b&=\sum_{h\in S}a_{h}t^{h}\cdot\sum_{h\in S'}b_{h}t^{h}=\sum_{l\in S+S'}\left(\sum_{\substack{h\in S \\ h'\in S' \\ h+h'=l}}a_{h}b_{h'}\sigma(-\beta(h,h'))\right)t^{l}=\\&=\sum_{l\in S+S'}\left(\sum_{\substack{h\in S \\ h'\in S' \\ h+h'=l}}b_{h'}a_{h}\sigma(-\beta(h',h))\right)t^{l}= \\&=\sum_{h\in S'}b_{h}t^{h}\cdot\sum_{h\in S}a_{h}t^{h}=b\cdot a.
\end{align*}
\item \textit{Distributivity.} \begin{align*}a\cdot(b+c)&=\sum_{h\in S}a_{h}t^{h}\left(\sum_{h\in S'\cup S''}(b_{h}+c_{h})t^{h}\right)=\\&=\sum_{l\in S+(S'\cup S'')}\left(\sum_{\substack{h\in S \\h'\in S'\cup S''\\h+h'=l}}a_{h}(b_{h'}+c_{h'})\sigma(-\beta(h,h'))\right)t^{l}=\\&=\sum_{l\in S+(S'\cup S'')}\left(\sum_{\substack{h\in S \\h'\in S'\cup S''\\h+h'=l}}a_{h}b_{h'}\sigma(-\beta(h,h'))+a_{h}c_{h'}\sigma(-\beta(h,h'))\right)t^{l}=\end{align*}\begin{align*}&=\sum_{l\in S\cup S'}\left(\sum_{\substack{h\in S\\h'\in S'\\h+h'=l}}a_{h}b_{h'}\sigma(-\beta(h,h'))\right)t^{l}+\sum_{l\in S\cup S''}\left(\sum_{\substack{h\in S\\h''\in S''\\h+h''=l}}a_{h}c_{h''}\sigma(-\beta(h,h''))\right)t^{l}=\\&=\left(\sum_{h\in S}a_{h}t^{h}\cdot\sum_{h\in S'}b_{h}t^{h}\right)+\left(\sum_{h\in S}a_{h}t^{h}\cdot\sum_{h\in S''}c_{h}t^{h}\right)=\\&=a\cdot b+a\cdot c.\end{align*}
\item \textit{Unit.} Clearly $1\in K(t^{H},\beta)_{\delta}$ is the series such that $1f=f1=f$ for all $f\in K(t^{H},\beta)_{\delta}$.
\item \textit{Multiplicative inverse.} Let $b\in K(t^{H},\beta)_{\delta}$, $b\neq0$. The existence of the inverse element is given by the following arguments: \begin{enumerate}\item Let $a\in K(t^{H},\beta)$ be an element with strictly positive support (i.e. $min\,supp(a)>0$). One can prove, as in the usual power series case, that the sum $1+a+a^{2}+a^{3}+\ldots$ is an element of $K(t^{H},\beta)$ (see \cite{Neumann49}); \item Every element $b\in K(t^{H},\beta)$ can be written as $ct^{\tilde{h}}(1-a)$, where $c\in K$, $\tilde{h}\in H$, and $a\in K(t^{H},\beta)$ has strictly positive support.\\ \smallskip Indeed, let $b=\sum_{h\in S}b_{h}t^{h}$. Then $$b=ct^{\tilde{h}}(1-a),$$ where $\tilde{h}=min\,supp(b)$, $c=b_{\tilde{h}}$, and $a=\sum_{h\in S'}a_{h}t^{h}$ is such that $$a_{h}=-\frac{b_{h+\tilde{h}}}{c\sigma(-\beta(h,\tilde{h}))}.$$ Note that if $h<0$, then $h+\tilde{h}<\tilde{h}=min\,supp(b)$ and $b_{h+\tilde{h}}=0$. Thus, $a_{h}=0$ and $a$ has strictly positive support. \item Now, one can compute the multiplicative inverse of $b$ as follows $$b^{-1}=(ct^{\tilde{h}}(1-a))^{-1}=\frac{c^{-1}t^{-\tilde{h}}}{\sigma(-\beta(\tilde{h},-\tilde{h}))}(1+a+a^{2}+a^{3}+\ldots).$$\end{enumerate} 
\end{itemize} \end{proof}
We show now that $K(t^{H},\beta)_{\delta}$ can be equipped with a non-trivial valuation as follows.

\begin{prop} \label{valuation} The field $K(t^{H},\beta)_{\delta}$ can be equipped with a valuation defined as follows $$val: K(t^{H},\beta)_{\delta}\longrightarrow G_{\beta}\cup\{(\infty,\infty)\}$$ such that $$val(\sum_{h\in S}a_{h}t^{h})=(v(a_{h_{0}}),h_{0})$$ where $v$ is the mixed characteristic valuation over $K$, $h_{0}$ is the smallest element of S and $(\infty,\infty)$ is a symbol for an element greater than any element in $G_{\beta}$ for the valuation of $0$. 
\end{prop}
\begin{proof} We check that $val$ is a valuation.
\begin{itemize}
\item Clearly $val(a)=(\infty,\infty)$ if and only if $a=0$;
\item Let $a,b\in K(t^{H},\beta)$ such that $h_{1}=min\,supp(a)$ and $h_{2}=min\,supp(b)$. Then $min\,supp(ab)=h_{1}+h_{2}$ and the corresponding  non-zero coefficient is $a_{h_{1}}b_{h_{2}}\sigma(-\beta(h_{1},h_{2}))$. We have \begin{align*} val(ab) & =(v(a_{h_{1}}b_{h_{2}}\sigma(-\beta(h_{1},h_{2}))),h_{1}+h_{2})=\\ & =(v(a_{h_{1}})+v(b_{h_{2}})+v(\sigma(-\beta(h_{1},h_{2}))),h_{1}+h_{2})=\\ &=(v(a_{h_{1}})+v(b_{h_{2}})-\beta(h_{1},h_{2}),h_{1}+h_{2})= \\ & =(v(a_{h_{1}}),h_{1})+v(a_{h_{2}}),h_{2})=val(a)+val(b);\end{align*}
\item Let $a,b\in K(t^{H},\beta)$ with $h_{1}=min\,supp(a)$ and $h_{2}=min\,supp(b)$. Then $val(a+b)=(v(a_{h}+b_{h}),h)$, where $h=min\{h_{1},h_{2}\}$, if $h_{1}\neq h_{2}$ or $h=h_{1}=h_{2}$, otherwise. Then, $$val(a+b)\geq(min\{v(a_{h}),v(b_{h})\},h)=min\{val(a),val(b)\}.$$
\end{itemize} \end{proof}

We denote the valued field defined above with values in $G_{\beta}$ by $K_{\beta}$. The following result ensures that the construction of $K_{\beta}$ is independent of the choice of $\beta$.
\begin{prop} Let $\beta'$ be a $2$-cocycle defined by a different transversal map $\alpha':H\longrightarrow G$. Then $K_{\beta'}$ valued in $G_{\beta'}$ is isomorphic to $K_{\beta}$ as valued fields. \end{prop}
\begin{proof} By Fact \ref{relationbetas}, $\beta'=\beta-\psi^{*}$, where $\psi:H\longrightarrow A$ is a $1$-cochain. Thus $$\sigma(-\beta'(h,h'))=\sigma(-\beta(h,h'))\frac{\sigma(\psi(h+h'))}{\sigma(\psi(h))\sigma(\psi(h'))}.$$ Consider the map defined by linearity from\begin{align*}
f: & \,\,\,\,\,\,\,\,\,\,\,\,\,\,\,K_{\beta} \,\,\,\,\,\,\longrightarrow \,\,\,\,\,\,K_{\beta'} \\
   & \,\,\,\,\,\,\,\,\,\,\,\,\,\,\,a\in K \,\mapsto a\in K \\
   & \,\,\,\,\,\,\,\,\,\,\,\,\,\,\,\,\,\,\,\,\,\,\,t \,\,\,\,\,\,\,\mapsto \,\,\,\,t \\
   & \sigma(-\beta(h,h')) \mapsto \sigma(-\beta'(h,h'))
\end{align*} Then $f$ is an isomorphism of fields. Let $\hat{f}: G_{\beta}\longrightarrow G_{\beta'}$ be the isomorphism of groups such that $f(\beta)=\beta'$ and it is the identity elsewhere. Then $\hat{f}$ satisfies $\hat{f}(val(a))=val'(f(a))$ for all $a\in K_{\beta}$, where $val$ is the valuation on $K_{\beta}$ and $val'$ is the valuation on $K_{\beta'}$. Thus $f$ is an isomorphism of valued fields.
\end{proof}


\begin{thm} \label{pseudocompleteness} Let $\lambda$ be a limit ordinal. The valued field $(K(t^{H},\beta)_{\delta},val)$ is $\lambda$-pseudo complete if $K$ is $\lambda$-pseudo complete.\end{thm}
\begin{proof} 
Let $\lambda$ be a limit ordinal and $(c_{\alpha})_{\alpha<\lambda}$ be a pseudo-Cauchy sequence in $K(t^{H},\beta)_{\delta}$ and let $c_{\alpha}=\sum_{h\in S}a_{\alpha,h}t^{h}$. Since it is pseudo-Cauchy, there exists $\alpha_{0}<\lambda$ such that for all $\alpha_{0}<\alpha'<\alpha*<\alpha''$, $$val(c_{\alpha''}-c_{\alpha*})>val(c_{\alpha*}-c_{\alpha'}).$$ Set $g_{\alpha}:=val(c_{\alpha}-c_{\alpha+1})=(v(a_{\alpha,h_{\alpha,\alpha+1}}-a_{\alpha+1,h_{\alpha,\alpha+1}}),h_{\alpha,\alpha+1}),$ where $h_{\alpha,\alpha+1}$ is the minimum of $supp(c_{\alpha}-c_{\alpha+1})$. Since $$val(c_{\alpha}-c_{\alpha'})\geq min\{val(c_{\alpha}-c_{\alpha+1}),val(c_{\alpha+1}-c_{\alpha'})\},$$ and the sequence is pseudo-Cauchy, we have that for all $\alpha'>\alpha>\alpha_{0}$, $$val(c_{\alpha}-c_{\alpha'})=val(c_{\alpha}-c_{\alpha+1})=g_{\alpha}.$$ Thus, if we set $h_{\alpha,\alpha+1}=\dot{h}_{\alpha}$ and $v(a_{\alpha,h_{\alpha,\alpha+1}}-a_{\alpha+1,h_{\alpha,\alpha+1}})=\dot{\gamma}_{\alpha}$, then for all $\alpha'>\alpha>\alpha_{0}$, $$h_{\alpha,\alpha'}=\dot{h}_{\alpha}$$ and $$v(a_{\alpha,h_{\alpha,\alpha'}}-a_{\alpha',h_{\alpha,\alpha'}})=v(a_{\alpha,\dot{h}_{\alpha}}-a_{\alpha',\dot{h}_{\alpha}})=\dot{\gamma}_{\alpha}.$$ It follows that for all $\alpha'>\alpha>\alpha_{0}$, $$a_{\alpha,h}=a_{\alpha',h},\,\,\text{for all}\,\,h<\dot{h}_{\alpha},$$ and the difference of their $\dot{h}_{\alpha}$-coefficients has always a fixed valuation $\dot{\gamma}_{\alpha}$ depending on $\alpha$. Since the sequence is pseudo-Cauchy, the sequence of valuations $g_{a}=(\dot{\gamma}_{a},\dot{h}_{\alpha})_{\alpha<\omega}$ for all $\alpha>\alpha_{0}$ is strictly increasing. We have two cases: \begin{enumerate} \item $(\dot{h}_{\alpha})_{\alpha<\lambda}$ is strictly increasing for all $\alpha>\alpha_{0}$, \item there is $\alpha_{1}>\alpha_{0}$ such that for all $\alpha>\alpha_{1}$, $\dot{h}_{\alpha}=\dot{h}_{\alpha_{1}}$ and $(\dot{\gamma}_{\alpha})_{\alpha_{1}<\alpha<\lambda}$ is strictly increasing.\end{enumerate}
If case (1) holds, let $$b_{h}:=\left \{ \begin{array}{rl}
a_{\alpha,h}\,\,\,\,\,, & \text{if}\,\,h<\dot{h}_{\alpha}\,\,\text{for some}\,\alpha<\omega;\\
0\,\,\,\,\,, & \,\,\text{otherwise.}
\end{array}
\right.$$
If case (2) holds, let $$b_{h}:=\left \{ \begin{array}{rl}
a_{\alpha_{1},h}\,\,\,\,\,, &\,\, \text{if}\,\,h<\dot{h}_{\alpha_{1}},\\
PL\left((a_{\alpha,\dot{h}_{\alpha_{1}}})_{\alpha<\omega}\right), &\,\,\text{if}\,\,h=\dot{h}_{\alpha_{1}},\\
0\,\,\,\,\,, & \,\,\text{otherwise,}
\end{array}
\right.$$

where $PL\left((a_{\alpha,\dot{h}_{\alpha_{1}}})_{\alpha<\omega}\right)$ is any pseudo-limit of the sequence in brackets, whose existence is ensured by the pseudo completeness of $(K, v)$.

We now show that $$b:=\sum_{h\in H}b_{h}t^{h}\in K(t^{H},\beta)_{\delta},$$ that is, $supp(b)=\{h\in H\vert\,b_{h}\neq0\}$ is a well ordered subset of $H$.
It suffices to show that for all $\eta\in supp(b),$ $(-\infty,\eta]\cap supp(b)$ is well ordered. We have different arguments depending on the two above cases. Take $\eta\in supp(b)$,\begin{enumerate} 
\item if case (1) holds, there exists $\alpha<\omega$ such that $\eta<\dot{h}_{\alpha}$, thus $\eta\in supp(c_{\alpha})$. In particular, $$(-\infty,\eta]\cap supp(b)\subseteq supp(c_{\alpha})$$ and thus $$(-\infty,\eta]\cap supp(b)=supp(c_{\alpha})\cap(-\infty,\eta],$$ that is a well ordered subset of $H$.
\item if case (2) holds, then $\eta=\dot{h}_{\alpha_{1}}=\dot{h}_{\alpha}$ for all $\alpha>\alpha_{1}$. In particular, $\eta\in supp(c_{\alpha_{1}})$. Then the claim follows as in the previous case.
\end{enumerate}
Finally, by definition of $b$, $$val(b-c_{\alpha})=(v(b_{\dot{h}_{\alpha}}-a_{\alpha,\dot{h}_{\alpha}}),\dot{h}_{\alpha})=(v(a_{\alpha+1,\dot{h}_{\alpha}}-a_{\alpha,\dot{h}_{\alpha}}),\dot{h}_{\alpha})=(\dot{\gamma}_{\alpha},\dot{h}_{\alpha})= g_{\alpha}.$$ Since $g_{\alpha}$ is strictly increasing, then $(c_{\alpha})_{\alpha\in\omega}$ converges to $b$.
\end{proof}
Notice that, if $K$ is $\lambda$-pseudo complete for some limit ordinal $\lambda$, then $K(t^{H},\beta)_{\delta}$ is henselian.
 
In the following proposition, we show how the pseudo completeness of $K(t^{H},\beta)_{\delta}$ can be used to prove the existence of inverse elements for non-zero elements in $K(t^{H},\beta)_{\delta}$, and thus that it is a field. We use the method shown in \cite[Chapter 3]{EP05} for generalized power series, giving an alternative proof of Proposition \ref{field}.
\begin{prop}\label{inverse} Let $(K,v)$ be an $\omega$-pseudo complete valued field with values in $A$, the smallest convex subgroup of the value group of $K$. If $f\in K(t^{H},\beta)_{\delta}$ is non-zero, then it has a multiplicative inverse. \end{prop}
\begin{proof} Take $0\neq f=\sum_{h\in S}a_{h}t^{h}\in K(t^{H},\beta)_{\delta}$, with $h_{1}=min\,supp(f)$. Without loss of generality we can assume $a_{h_{1}}=1$. Consider the set $$\Sigma=\left\{val(1-fg)\,\vert\,g\in K(t^{H},\beta)_{\delta},\,1-fg\neq0\right\}.$$ First of all, we show that $\Sigma$ has no maximal element. For instance, take $g\in K(t^{H},\beta)_{\delta}$ such that $1-fg\neq0$. Then $1-fg=\sum_{h\in S'}b_{h}t^{h}$ and take $h_{2}=min\,supp(1-fg)$. Set $g'=b_{h_{2}}p^{\beta(h_{1},h_{2}-h_{1})}t^{h_{2}-h_{1}}$. Then $min\,supp(fg')=h_{2}$ and its first non-zero coefficient is $b_{h_{2}}$. Thus $$val(1-f(g+g'))=val((1-fg)-fg')>val(1-fg).$$ 

In this way we can choose a strictly increasing sequence $(\gamma_{\nu})_{\nu<\rho}$ of elements of $\Sigma$ that is cofinal in $\Sigma$, that is for every $\gamma\in\Sigma$ there is $\nu<\rho$ such that $\gamma<\gamma_{\nu}$. Moreover, we choose $(g_{\nu})_{\nu<\rho}$ a sequence of elements of $K(t^{H},\beta)_{\delta}$ such that $\gamma_{\nu}=val(1-fg_{\nu})$ for every $\nu$. Notice that, since $(\gamma_{\nu})_{\nu<\rho}$ is strictly increasing, for every $\rho>\mu>\nu>0$, $$val(1-fg_{\mu}-(1-fg_{\nu}))=val(1-fg_{\nu})=\gamma_{\nu}$$ and so $$\gamma_{\nu}=val(1-fg_{\mu}-(1-fg_{\nu}))=val(f(g_{\nu}-g_{\mu}))=val(f)+val(g_{\mu}-g_{\nu}).$$ Thus, for every $\rho>\mu>\nu>\lambda>0$, $$val(g_{\mu}-g_{\nu})>val(g_{\nu}-g_{\lambda})\,\,\text{if and only if}\,\,\gamma_{\nu}>\gamma_{\lambda}.$$ It follows that the sequence $(g_{\nu})_{\nu<\rho}$ is pseudo-Cauchy. 

Since $K(t^{H},\beta)_{\delta}$ is pseudo complete, let $r=\sum_{h\in S^{*}}c_{h}t^{h}$ be the pseudo-limit of $(g_{\nu})_{\nu<\rho}$. 

\textbf{Claim:} $1-fr=0$. Suppose by contradiction that there is $0<\nu<\rho$ such that $val(1-fr)<val(1-fg_{\nu})=\gamma_{\nu}$. Since $(\gamma_{\nu})_{\nu<\rho}$ is strictly increasing, then for every $0<\nu<\mu<\rho$ we have $$val(1-fr)<val(1-fg_{\nu})<val(1-fg_{\mu}).$$ Moreover, since $r$ is the pseudo-limit of $(g_{\nu})_{\nu<\rho}$, for every $0<\nu<\mu<\rho$, $val(r-g_{\mu})>val(r-g_{\nu})$. Then we have \begin{align*} val(1-fr)&=val(1-fg_{\nu}-(1-fr))=val(f(r-g_{\nu}))=\\&=val(f)+val(r-g_{\nu})<val(f)+val(r-g_{\mu})=\\&=val(f(r-g_{\mu}))=val(1-fg_{\mu}-(1-fr))=val(1-fr), \end{align*} which proves the contradiction. Thus $val(1-fr)>val(1-fg_{\nu})$ for every $0<\nu<\rho$ and since $\Sigma$ has no maximal element, $1-fr=0$.  \end{proof}

\section{A characterization for formally $p$-adic fields}
One can apply the above construction to the field of $p$-adic numbers as follows. Consider any $\mathbb{Z}$-group $G$. Since $\mathbb{Z}$ is a prime model of its theory, there is a unique elementary embedding of $\mathbb{Z}$ into $G$. Let $H=G/\mathbb{Z}$, then $H$ is a divisible ordered abelian group, since $\mathbb{Z}$ is convex in $G$. We have the following exact sequence $$0\xrightarrow{}\mathbb{Z}\xrightarrow{i} G\xrightarrow{\rho} H\xrightarrow{} 0.$$
\begin{rem} An exact sequence of $\mathbb{Z}$-groups as above does not always split, i.e. there is a $\mathbb{Z}$-group $G$ that does not have $\mathbb{Z}$ as a direct summand.\end{rem}
\begin{proof} See \cite[Corollary 2, p. 266]{Mendelson61}. In particular, it is shown that if $R$ is a non-standard model of Peano arithmetic (i.e. an ordered ring elementarily equivalent to $\mathbb{Z}$), then the additive group cannot have $\mathbb{Z}$ as a direct summand.\end{proof} 
Consider an exact sequence as above which does not split and define a cocycle $\beta$ from $H$ to $\mathbb{Z}$. Since $\mathbb{Q}_{p}$ with the $p$-adic valuation has mixed characteristic and values in $\mathbb{Z}$ we have the following corollary that summarizes the construction given by Ax and Kochen in \cite{AxK65}.
\begin{cor} The field $\mathbb{Q}_{p}(t^{H},\beta)_{\delta}$ defined as in Section \ref{twisted}, is an $\omega$-pseudo complete valued field with the valuation $val_{p}$ on $\mathbb{Z}\times H$ defined as in Proposition \ref{valuation}. \end{cor}

We recall the following theorem from Ax-Kochen's paper \cite{AxK65} that follows from the Hahn-like construction over $\mathbb{Q}_{p}$ and \cite[Theorem 1]{AxK65}.
\begin{thm}\label{AKchar} If $K$ is an $\omega$-pseudo complete formally $p$-adic field with value group a $\mathbb{Z}$-group $G$ of cardinality $\aleph_{1}$, then it is isomorphic to $\mathbb{Q}_{p}(t^{H},\beta)_{\aleph_{0}}$, where $H=G/\mathbb{Z}$, assuming the Continuum Hypothesis. \end{thm}

From the previous theorem, it follows that the valued field $\mathbb{Q}_{p}(t^{H},\beta)_{\aleph_{0}}$ is elementarily equivalent to $\mathbb{Q}_{p}$. This can be proved also using a different argument as shown in the following proposition.

\begin{prop} \label{elemequiv} The valued field $\mathbb{Q}_{p}(t^{H},\beta)_{\delta}$ is formally $p$-adic. \end{prop}
\begin{proof} We consider the conventional Hahn-construction on $\mathbb{Q}_{p}$ with values in $H$ and residue field $\mathbb{Q}_p$. This is the coarsening of an elementary extension of $\mathbb{Q}_{p}$, with values in the $\mathbb{Z}$-group $H\times\mathbb{Z}$, given by the composition of the $t$-adic and the $p$-adic valuation as in the examples of \cite[Section 7]{ADM1}. By Ax-Kochen/Ershov principle in the $(0,0)$ characteristic case, it is elementarily equivalent to the coarsening of the valuation defined over $\mathbb{Q}_{p}(t^{H},\beta)$ with values in $\mathbb{Z}\times H$. Then, $\mathbb{Q}_{p}(t^{H},\beta)$ and $\mathbb{Q}_{p}$ are elementarily equivalent as fields. Moreover $\mathbb{Q}_{p}(t^{H},\beta)$ is an henselian mixed characteristic valued field with $v(p)=(1,0)$, the minimal positive element of the value group. Thus, its valuation ring is existentially definable by Julia Robinson's formula. It follows that $\mathbb{Q}_{p}(t^{H},\beta)$ is formally $p$-adic.\end{proof}

\section{A general characterization}
In this section, we extend the characterization given in Theorem \ref{AKchar} to any finitely ramified henselian valued field of mixed characteristic with fixed discrete value group $G$ and residue field $k$. 
Now we fix $k$ a field of characteristic $p\geq2$ and a discretely ordered abelian group $G$ with minimal positive element $1_{G}$. Let $A$ be the minimal convex subgroup of $G$ generated by $1_{G}$, then $A\cong\mathbb{Z}$. 
In \cite[Section 5, Theorems 3-4]{S79} and \cite{cohen46}, the authors give a characterization of complete unramified valued fields with residue field $k$  and value group $\mathbb{Z}$, respectively in the following two cases: \begin{itemize} \item as the field of fraction $W(k)$ of the ring of Witt vectors over $k$, if $k$ is perfect; \item as the field of fraction $C(k)$ of the Cohen ring over $k$, if $k$ is not perfect. \end{itemize} Since in this paper we will not generally distinguish between the two cases, we will use the same notation $C(k)$ for both. Thus, if $k$ is any field, $C(k)$ is a complete unramified valued field with values in $\mathbb{Z}$. Moreover, for every $e\geq 1$, there exists $\pi$ in the algebraic closure $\overline{C(k)}$ which is a root of an Eisenstein polynomial of degree $e$ over $C(k)$. Thus, $C(k)(\pi)$ is a complete valued field with ramification $e$, and its value group is an ordered group $A\cong\mathbb{Z}$. Let $$0\xrightarrow{}A\xrightarrow{i} G \xrightarrow{\rho} H\xrightarrow{} 0$$ be the corresponding exact sequence of ordered abelian groups, and $\beta: H\times H\longrightarrow A$ a $2$-cocycle. Then, the field $C(k)(\pi)(t^{H},\beta)$ is a mixed characteristic valued field with value group $G_{\beta}$, as we show in the following lemma.
\begin{lem} The valued field $\left(C(k)(\pi)(t^{H},\beta)_{\delta},val\right)$ has ramification index $e$ and residue field $k$. \end{lem}
\begin{proof} Set $$C(k)[t^{H},\beta]_{\delta}=\{x\in C(k)(t^{H},\beta)_{\delta}\,\vert\, min\,supp(x)\geq0\},$$ and denote by $(t)^{\geq0}$ the ideal of the ring $C(k)[t^{H},\beta]_{\delta}$ generated by $t$. The valuation ring of $(C(k)(t^{H},\beta)_{\delta},val)$ is \begin{align*}O_{val}&=\left\{x=\sum a_{h}t^{h}\in C(k)(t^{H},\beta)\,\vert\,val(x)\geq0\right\}=\\&=\left\{x=\sum a_{h}t^{h}\in C(k)(t^{H},\beta)_{\delta}\,\vert\,h_{0}>0\,\,\text{or}\,\,h_{0}=0\,\wedge\,v(a_{h_{0}})\geq0,\,h_{0}=min\,supp(x)\right\}=\\&=\left\{x=\sum a_{h}t^{h}\in C(k)[t^{H},\beta]_{\delta}\,\vert\,x\in(t)^{\geq0}\,\,\text{or}\,\,v(a_{0})\geq0\right\}=\\&=C[k]+(t)^{\geq0},\end{align*} where $C[k]$ is the Cohen ring over $k$ and so the valuation ring of $C(k)$. Thus, the maximal ideal of $O_{val}$ is $$M_{val}=\left\{x=\sum a_{h}t^{h}\in C(k)[t^{H},\beta]_{\delta}\,\vert\,x\in(t)^{\geq0}\,\,\text{or}\,\,v(a_{0})>0\right\}=M+(t)^{\geq0},$$ where $M$ is the maximal ideal of $C[k]$. Hence the residue field is given by $$C[k]+(t)^{\geq0}/M+(t)^{\geq0}\cong k.$$ Moreover, $val(\pi)=(1,0)$, and $$val(p)=(e,0)=e(1,0)=(1,0)+\ldots+(1,0)=(e-(e-1)\beta(0,0),0)=(e,0),$$ thus $val(p)$ is $e$ times the minimal positive element of $G_{\beta}$, since $\beta(0,0)=0$.\end{proof}
\begin{lem} Let $L$ be a finite extension of $C(k)$, then $L$ is pseudo complete.\end{lem}
\begin{proof} The value group of $L$ is isomorphic to $\mathbb{Z}$, hence every pseudo-Cauchy sequence is a Cauchy sequence. So, completeness for both $L$ and $C(k)$ is equivalent to be pseudo complete. Now, the result follows from \cite[Corollary 2, p.57]{CF67}, which states that finite extensions of complete valued fields are complete with respect to the extended valuation. \end{proof}
\begin{cor} The valued field $C(k)(t^{H},\beta)_{\delta}$ is pseudo complete. \end{cor}
\begin{proof} It follows directly from Theorem \ref{pseudocompleteness}. \end{proof}
We are now in a position to generalize Ax and Kochen result for formally $p$-adic fields in \cite{AxK65i} and \cite{AxK65} to finitely ramified mixed characteristic valued fields.
First of all, we note the following
\begin{rem} \label{embeddingqp} Let $K$ be an $\omega$-pseudo complete unramified mixed characteristic valued field with residue field $k$. Then there is an embedding (not necessarily elementary) of valued fields from $C(k)$ into $K$. \end{rem}
\begin{proof} We can apply the coarsening method to $K$ and consider the core field $\mathring{K}$, that is the residue field of $K$ with respect to the coarse valuation. The core field $\mathring{K}$ is equipped with a mixed characteristic valuation with residue field $k$ and values in the smallest convex subgroup of the value group of $K$. Since $K$ is pseudo complete then clearly also the core field is pseudo complete. In particular, since the smallest convex subgroup of its value group is isomorphic to $\mathbb{Z}$, the core field is in particular complete. Then by the characterization of complete unramified valued fields with a fixed residue field, $\mathring{K}$ is isomorphic to $C(k)$, and thus $C(k)$ is a valued subfield of $K$.\end{proof}
If $K$ is finitely ramified, we have the following
\begin{lem} \label{C} Let $K$ be a pseudo complete mixed characteristic valued field with finite ramification $e\geq1$, and residue field $k$. Then $K$ contains an isomorphic copy of a finite extension of $C(k)$ of degree $e$. \end{lem}
\begin{proof} Let $\pi\in K$ be a uniformizer. As before we can apply the coarsening method. By \cite{CF67} and \cite[p.27]{PR84}, the core field $\mathring{K}$ is isomorphic to the finite extension $C(k)(\pi)$ of the Cohen field. Thus, $C(k)(\pi)$ is a valued subfield of $K$. Moreover, since the element $\pi$ is a uniformizer for $K$, then it is the root of an Eisenstein polynomial $p(x)$ of degree $e$ with coefficients in $C(k)$. \end{proof}
Note that the assumption of $K$ being pseudo complete in Lemma \ref{C} is necessary. Indeed, as a counterexample, in \cite{BvDDM88} the authors show that for every finite extension $L$ of $\mathbb{Q}_{p}$, there exists a valued field $K$ elementarily equivalent to $L$, that is not pseudo complete, such that there is no subfield of $K$ elementarily equivalent to $\mathbb{Q}_{p}$. Moreover, notice that if $K$ is a valued field as in Lemma \ref{C}, and $K'$ is another pseudo complete mixed characteristic valued field with finite ramification $e$, residue field $k$, and uniformizer any root of $p(x)$, there is an isomorphic copy of $C(k)(\pi)$ inside $K'$. 
For simplicity of notation, we introduce the following property. 

Let $K, K'$ be two pseudo complete finitely ramified henselian valued fields. We say that $K, K'$ have the property \textbf{P} if \begin{itemize} \item $K, K'$ have same residue field $k$; \item if $\pi,\pi'$ are two uniformizers in $K, K'$ respectively, then they are conjugate, i.e. roots of the same Eisenstein polynomial $p(x)\in C(k)$.\end{itemize}
We now prove the analogue of Proposition \ref{elemequiv} for $C(k)(t^{H},\beta)_{\delta}$.
\begin{prop} The valued field $C(k)(\pi)(t^{H},\beta)_{\delta}$ is elementarily equivalent to $C(k)(\pi)$ as valued fields. \end{prop}
\begin{proof} We use the same argument of Proposition \ref{elemequiv} for formally $p$-adic fields. Indeed, we have the assert by coarsening, AKE principle in the characteristic $(0,0)$ case, and the definability of the valuation (see \cite{CDLM13}). Moreover, $C(k)(\pi)(t^{H},\beta)_{\delta}$ and $C(k)(\pi)$ have property \textbf{P}, which ensures the assert in the case of finite ramification (see \cite[Example 12.8]{AJ19cohen} for a counterexample of mixed characteristic valued field with finite ramification and not having property \textbf{P} that are not elementarily equivalent).  \end{proof}

We need the following lemmas in order to prove Theorem \ref{mainthm}.
\begin{lem} \label{relalgclosed} Let $K$ be a mixed characteristic pseudo complete valued field with finite ramification. If $F$ 
is a henselian valued subfield of $K$ with no immediate extensions in $K$ and its value group $\Gamma_{F}$ is a pure countable subgroup of $\Gamma_{K}$, then $F$ is relatively algebraically closed in $K$.\end{lem}
\begin{proof} Assume $F$ has a proper finite extension in $K$. Then, by 
purity of $\Gamma_{F}$ into $\Gamma_{K}$, this extension is immediate. This contradicts the hypothesis.\end{proof}

On the other hand, we have the following
\begin{prop} \label{rac-hens} Let $(K, v_{K})$ be a henselian mixed characteristic valued field with finite ramification. 
If $F$ is a relatively algebraically closed subfield of $K$, then $F$ is henselian. \end{prop}
\begin{proof} 
Let $v_{F}$ be the valuation induced by $v_{K}$ to $F$. We prove that $v_{F}$ extends uniquely to any finite extension. Let $\alpha$ be an element in the algebraic closure $\overline{F}$ of $F$, and $v_{1}, v_{2}$ be two extensions of $v_{F}$ to $F(\alpha)$ such that $v_{1}(\alpha)\neq v_{2}(\alpha)$. Since $F$ is relatively algebraically closed in $K$, the field $K(\alpha)$ is an algebraic extension of $K$. Then $v_{1}$ and $v_{2}$ extend to two valuations $v_{1}'$ and $v_{2}'$ over $K(\alpha)$ such that $v_{1}'(\alpha)\neq v_{2}'(\alpha)$ and $v^{'}_{1\restriction_{K}}=v^{'}_{2\restriction_{K}}=v_{K}$. This contradicts henselianity of $K$.\end{proof}
\begin{lem} \label{lem-prop1} Let $(K, v_{K})$ be a mixed characteristic valued field with finite ramification. Let $F$ be a relatively algebraically closed subfield of $K$. If $\alpha\in\overline{F}$ and $y\in K$, then there exists $r\in F$ such that $v(r-y)\geq v(\alpha-y)$, where $v$ is the extension of $v_{K}$ to a proper algebraic extension.\end{lem}
\begin{proof} Let $\alpha\in\overline{F}$, then $[F(\alpha):F]=m$ for some $m>1$, and $m=e'f'$, for some $e',f'\geq1$, since $F$ is henselian by Proposition \ref{rac-hens}. Then by 
the characterizations of unramified and totally ramified extensions (see \cite[Chapter I]{CF67}), if $k$ is the residue field of $F$, there exist $\xi\in \overline{k}$ and $\pi'$ a uniformizer for $F(\alpha)$, such that $F(\alpha)=F(\xi,\pi')$. Set $\theta_{i,j}=\xi^{i}\pi^{j}$ for $i=0,\ldots,f'-1$ and $j=0,\ldots,e'-1$. Then, $\theta_{i,j}$'s form a basis of $F(\alpha)$ over $F$, and so there are $f_{i,j}\in F$ such that $$\alpha=\sum_{i,j}f_{i,j}\theta_{i,j}.$$ By a standard argument (see \cite[p. 57]{S79}), for every $f_{i,j}\in F$ we have $$v_{F(\xi,\pi')}(\sum_{i,j}f_{i,j}\theta_{i,j})=min_{i,j}\{v_{F(\xi,\pi')}(f_{i,j}\theta_{i,j})\},$$ that means that $\theta_{i,j}$'s form a separated basis for the extension $F(\xi,\pi')/F$
. Then, the proof follows as in \cite[Proposition 1]{AxK65i}. Indeed, setting $$d_{i,j}:=\left \{ \begin{array}{rl}
1\,\,\,\,\,, & \text{if}\,\,i=j=0;\\
0\,\,\,\,\,, & \,\,\text{otherwise.}
\end{array}
\right.$$ we have $$\alpha-y=\sum_{i,j}(f_{i,j}-d_{i,j}y)\theta_{i,j}.$$ Now, let $v$ be the extension of $v_{K}$ to $K(\xi,\pi')$. Then, we have $$v(\alpha-y)=min_{i,j}((f_{i,j}-d_{i,j}y)\theta_{i,j})\leq v(f_{0,0}-d_{0,0}y)=v(f_{0,0}-y).$$ Thus, $r:=f_{0,0}$ is an element of $F$ with the required property.  \end{proof}

\begin{lem} \label{ax1} Let $K$ be a pseudo complete value field with finite ramification $e\geq1$, and with value group $\Gamma_{K}$. 
If $F$ is a valued subfield of $K$ such that 
$F$ has no immediate extensions in $K$ and its value group $\Gamma_{F}$ is a pure countable subgroup of $\Gamma_{K}$, then $F$ is pseudo complete. \end{lem}
\begin{proof} Using Lemma \ref{C} and \ref{lem-prop1}, we can now adapt the proof of \cite[Proposition 2]{AxK65i}, to the field $K$. We need to show that every pseudo-Cauchy sequence of $F$ has a limit in $F$. Since $\Gamma_{F}$ is countable, it suffices to prove that $F$ is $\omega$-pseudo complete. Let $(c_{i})_{i<\omega}$ be a pseudo-Cauchy sequence. Since $K$ is pseudo complete, there is $c\in K$ that is a pseudo-limit. If there exists $f\in F$ such that \begin{equation}\label{eq} v_{K}(f-c)>v_{K}(c-c_{i}),\,\end{equation} for all $i<\omega$, then $$v_{F}(f-c_{i})=v_{K}(f-c_{i})=v_{K}((f-c)+(c-c_{i}))=v_{K}(c-c_{i})$$ for all $i<\omega$, and so $f$ is a pseudo limit for $(c_{i})_{i<\omega}$ in $F$. Thus, suppose there is no $f\in F$ such that (\ref{eq}) holds. 
\smallskip

\textbf{Claim.} $c\in F$. We show that $F(c)$ is an immediate extension of $F$ in $K$. It suffices to show that $(c_{i})_{i<\omega}$ is of transcendental type over $F$. Indeed, if so, then for every polynomial $A(x)\in F[x]$, the sequence $(v_{k}(A(c_{i})))_{i<\omega}$ is eventually constant and $v_{K}(A(c))$ is its eventual value. Notice that if we factorize $A(x)$ over $\overline{F}$, it suffices to show that $v_{\overline{F}}(c-\alpha)$ is the eventual value of $(v_{\overline{F}}(c_{i}-\alpha))_{i<\omega}$, for every $\alpha$ root of $A(x)$ in $\overline{F}$. By Remark \ref{relalgclosed}, we know that $F$ is relatively algebraically closed in $K$. Hence, since $c\in K$, by Lemma \ref{lem-prop1} there exists $r\in F$ such that $v_{\overline{F}}(r-c)\geq v_{\overline{F}}(\alpha-c)$. Now, since the sequence is pseudo-Cauchy, there exists $i_{0}<\omega$ such that, for all $i>i_{0}$, $$v_{\overline{F}}(c-c_{i})>v_{\overline{F}}(c-c_{i-1})\geq v_{\overline{F}}(c-r)\geq v_{\overline{F}}(c-\alpha).$$ Thus, if $i>i_{0}$, $$v_{\overline{F}}(c_{i}-\alpha)=min\{v_{\overline{F}}(c-c_{i}),v_{\overline{F}}(c-\alpha)\}=v_{\overline{F}}(c-\alpha).$$ So, $v_{\overline{F}}(c-\alpha)$ is the eventual value of $(v_{\overline{F}}(c_{i}-\alpha))_{i<\omega}$. This completes the proof. \end{proof}

We recall the following definition from \cite{AxK65i}.

\begin{defn} If $H$ is a subgroup of an abelian group $G$, then $$(H\,\vert\,G)=\left\{x\in G\,\vert\,\text{there exists}\,m\in\mathbb{N}\,\text{such that}\,\,mx\in H\right\}.$$ \end{defn}
\begin{defn} Let $K$ be a valued field with finite ramification $e$, and let $\pi$ be a uniformizer for $K$. If $F$ is a subfield of $K$ such that 
$F$ and $K$ have property \textbf{P}, then $[F\,\vert\,K]$ denotes some fixed subfield $L$ of $K$ maximal with respect to the property $\Gamma_{L}\subseteq(\Gamma_{F}\,\vert\,G)$.\end{defn}

\begin{lem} \label{prop3ak2} Let $K,K'$ be two $\omega$-pseudo complete valued fields with common value group $G$ and satisfying property \textbf{P}. Let $F\subseteq K$ (resp. $F'\subseteq K'$) be a valued subfield with no immediate extensions in $K$ (resp. $K'$), and suppose $\Gamma_{F}$ is a countable, pure subgroup of $G$. Let $\psi: F\longrightarrow F'$ be an isomorphism of valued fields. If $x\in K\setminus F$, then there exists an isomorphism of valued fields $$\psi_{0}:[F(x)\,\vert\,K]\longrightarrow F_{1}^{'},$$ that extends $\psi$ and where $F_{1}^{'}$ is a subfield of $K'$ containing $F'$. Also, $$\Gamma_{[F(x)\,\vert\,K]}=(\Gamma_{F(x)}\,\vert\,G)=(\Gamma_{F_{1}^{'}}\,\vert\,G).$$\end{lem}
\begin{proof} The proof of \cite[Proposition 3]{AxK65} can be easily adapted to our case. It suffices to note that \begin{itemize} \item $\pi$ is a uniformizer for $K$ and, in general, $\pi\neq p$; \item by Lemma \ref{C}, we may assume $C(k)(\pi)\subseteq F$. \end{itemize} Then by 
Lemma \ref{lem-prop1}, Lemma \ref{ax1}, and the uniqueness of the immediate extension we have the assert. \end{proof}

\begin{thm}\label{mainthm} Let $K, K'$ be $\omega$-pseudo complete finitely ramified valued fields of cardinality $\aleph_{1}$ satisfying property \textbf{P}. If $K$ and $K'$ have the same value group $G$ of cardinality $\aleph_{1}$, then there exists an isomorphism of valued fields $$\psi: K\longrightarrow K'.$$ Moreover if $F, F'$ are two subfields of $K$ and $K'$ respectively and $\phi: F\longrightarrow F'$ is an isomorphism of valued fields, then $\psi$ is an extension of $\phi$. \end{thm}
\begin{proof} We follow the scheme of the proof of \cite[Theorem 2]{AxK65i}. Let $k$ be the common residue field of $K$ and $K'$. Then, by Lemma \ref{embeddingqp}, we may consider $C(k)$ as a common subfield of $K$ and $K'$. Now, let $\pi\in K$ and $\pi'\in K'$ be uniformizers of $K$ and $K'$, respectively. Since $K, K'$ satisfy property $P$, by Lemma \ref{C} $C(k)(\pi)$ and $C(k)(\pi')$ are two isomorphic valued subfields of $K$ and $K'$ respectively. Thus, the induced valuations have the same residue field $k$ and the same value group $\Gamma\cong\mathbb{Z}$, a countable pure subgroup of $G$. Now, choose two trascendence basis $B$ and $B'$ for $K$ and $K'$ over $C(k)(\pi)$ and $C(k)(\pi')$, respectively. Since $G$ has cardinality $\aleph_{1}$, it follows that $|B|=|B'|=\aleph_{1}$. Well order $B$ and $B'$ to form $\omega_{1}$-sequences. Thus, we have $B=(x_{\eta})_{\eta<\omega_{1}}$ and $B'=(x_{\eta}')_{\eta<\omega_{1}}$. For each $0<\lambda<\omega_{1}$, we define \begin{itemize} \item a subfield $F_{\lambda}\subseteq K$ with no immediate extensions in $K$ and with value group $\Gamma_{\lambda}$ a pure countable subgroup of $G$; \item a subfield $F_{\lambda}'\subseteq K'$ with no immediate extensions in $K'$ and with value group $\Gamma_{\lambda}'$ a pure countable subgroup of $G$; \item an isomorphism of valued fields $\psi_{\lambda}:F_{\lambda}\longrightarrow F_{\lambda}'$. \end{itemize}
For $\lambda=1$, $F_{1}:=C(k)(\pi)$ and $F^{'}_{1}=C(k)(\pi')$. Thus we have an isomorphism of valued fields $\psi_{1}:F_{1}\longrightarrow F^{'}_{1}$ such that $\psi_{1}(\pi)=\pi'$.
\\ Suppose $\lambda<\omega_{1}$ is an odd ordinal and we have defined an isomorphism of valued fields $\psi_{\lambda}:F_{\lambda}\longrightarrow F_{\lambda}^{'}$. Since $\Gamma_{F_{\lambda}}$ is countable, there exists $\mu<\omega_{1}$ such that $x_{\mu}\notin F_{\lambda}$. Let $\mu(\lambda)$ be the smallest such $\mu$. Let $$F_{\lambda+1}=[F(x_{\mu(\lambda)})\,\vert\,K].$$ Then, by Lemma \ref{prop3ak2}, there exists a valued subfield $F^{'}_{\lambda+1}\subseteq K'$ such that $\psi_{\lambda}$ can be extended to an isomorphism of valued fields $$\psi_{\lambda+1}:F_{\lambda+1}\longrightarrow F_{\lambda+1}^{'}.$$ If $\lambda$ is an even ordinal, we proceed in the same way, inverting the roles of $F_{\lambda}$ and $F^{'}_{\lambda}$ and considering $\psi^{-1}_{\lambda}$ instead of $\psi_{\lambda}$.
\\ If $\sigma<\omega_{1}$ is a limit ordinal, we consider the fields $D_{\sigma}=\bigcup_{\lambda<\sigma}F_{\lambda}$ and $D_{\sigma}^{'}=\bigcup_{\lambda<\sigma}F_{\lambda}^{'}$. Then there exists an isomorphism of valued fields $\phi:D_{\sigma}\longrightarrow D_{\sigma}^{'}$ extending all the $\psi_{\lambda}$, $\lambda<\sigma$. If we set \begin{align*}F_{\sigma}=&[D_{\sigma}\,\vert\,K], \\ F_{\sigma}^{'}=&[D_{\sigma}^{'}\,\vert\,K']\end{align*} then, by 
the uniqueness of the immediate extension, there is an isomorphism of valued fields $$\psi_{\sigma}: F_{\sigma}\longrightarrow F_{\sigma}^{'},$$ since $F_{\sigma}$ (resp. $F^{'}_{\sigma}$) is a maximal immediate extension of $D_{\sigma}$ (resp. $D_{\sigma}^{'}$). We have thus inductively proved that $W=\bigcup_{\lambda<\omega_{1}}F_{\lambda}$ and $W'=\bigcup_{\lambda<\omega_{1}}F^{'}_{\lambda}$ are isomorphic valued fields. By alternating the isomorphisms between $F_{\lambda}$ and $F^{'}_{\lambda}$ for even and odd $\lambda$, we have guaranteed that all the elements of the transcendence basis $B$ and $B'$ have been considered. Hence $K$ and $K'$ are algebraic over $W$ and $W'$, respectively. On the other hand, since each $F_{\lambda}$ (resp. $F^{'}_{\lambda}$) is relatively algebraically closed in $K$ (resp.: $K'$), then $W$ (resp. $W'$) is also relatively algebraically closed in $K$ (resp. $K_{1}'$). Hence, $K=W$ and $K'=W'$, and the theorem is proved.

\end{proof}
We now prove the following corollary, which characterizes $\omega$-pseudo complete finitely ramified valued fields with fixed residue field and value group.
\begin{cor}\label{finextscharacterization} Let $K$ be an $\omega$-pseudo complete finitely ramified valued field of cardinality $\aleph_{1}$, with residue field $k$, and valued in a group $G$ of cardinality $\aleph_{1}$. Then $K$ is isomorphic (as a valued field) to $L(t^{H},\beta)_{\aleph_{0}}$, where $L$ is a finite extension of the Cohen field $C(k)$ over $k$, $H=G/A$ and $\beta:H\times H\longrightarrow A$ a $2$-cocycle, modulo the Continuum Hypothesis.\end{cor}
\begin{proof} The proof follows directly from the $\omega$-pseudo completeness of $L(t^{H},\beta)_{\aleph_{0}}$ and from Theorem \ref{mainthm} applied to $K$ and $L(t^{H},\beta)_{\aleph_{0}}$, where $L$ is a finite extension of $C(k)$ such that $L$ and $K$ have the property \textbf{P}. \end{proof}

\section{Remarks and consequences} \label{cons}
\subsection{Lifting automorphisms}
One consequence of having an explicit description of a valued field in terms of power series, is the existence of an explicit lift of the automorphisms of the residue field and the value group to automorphisms of the valued field. For instance, in the equicharacteristic case, if $k$ is a field and $G$ an ordered abelian group, the Hahn-field $k((t^{G}))$ is naturally equipped with an automorphism $\phi(\sum_{g}a_{g}t^{g})=\sum_{g}\overline{\phi}(a_{g})^{\hat{\phi}(g)}$, where $\overline{\phi}\in Aut(k)$ and $\hat{\phi}\in Aut(G)$. In the mixed characteristic case, we notice that the same situation occurs for the Hahn-like construction $K(t^{H},\beta)_{\aleph_{0}}$, with some necessary adjustments. 
Indeed, in this case, the valuation is the composition of the $t$-adic valuation and the valuation on the field $K$ of the coefficients. Thus, a lifting of the automorphisms of the residue field and the value group is possible if they respect the cross-section of $K$, in the sense of (\ref{condition}) in Lemma \ref{respectval}. Recall that if $G$ is an ordered abelian group with smallest convex subgroup $A$ and $H:=G/A$, then $G_{\beta}$ defined as in Lemma \ref{lemmagroup}, is an ordered abelian group isomorphic to $G$. 
\begin{lem}\label{phioverbeta} Let $\hat{\phi}_{\beta}\in Aut(G_{\beta})$. There exist two endomorphisms $\hat{\phi}_{1}$ of $A$ and $\hat{\phi}_{2}$ of $H$, such that $$\hat{\phi}_{1}(-\beta(h,h'))=-\beta(\hat{\phi}_{2}(h),\hat{\phi}_{2}(h')).$$\end{lem}
\begin{proof} From the definition of $G_{\beta}$, there are two maps $\hat{\phi}_{1}:A\longrightarrow A$ and $\hat{\phi}_{2}:H\longrightarrow H$ such that $\hat{\phi}_{\beta}(z,h)=(\hat{\phi}_{1}(z), \hat{\phi}_{2}(h))$ for all $(z,h)\in G_{\beta}$. Since $\hat{\phi}_{\beta}$ is an automorphism of $G_{\beta}$, we have that $\hat{\phi}_{1}$ and $\hat{\phi}_{2}$ satisfy the following properties:
\begin{equation*}\begin{cases} \hat{\phi}_{1}(z)=0, \,\,\,\text{if}\, z=0; \\ \hat{\phi}_{2}(h)=0, \,\,\,\text{if}\, h=0; \\ \hat{\phi}_{1}(z+z'-\beta(h,h'))=\hat{\phi}_{1}(z)+\hat{\phi}_{1}(z')-\beta(\hat{\phi}_{2}(h),\hat{\phi}_{2}(h')), \,\,\,\text{for all}\, (z,h)\in G_{\beta};\\ 
\hat{\phi}_{2}(h+h')=\hat{\phi}_{2}(h)+\hat{\phi}_{2}(h'), \,\,\,\text{for all}\, h\in H.\end{cases}\end{equation*}
In particular, we have that $\hat{\phi}_{1}, \hat{\phi}_{2}$ are surjective endomorphisms of $A$ and $H$ respectively, and that $$\hat{\phi}_{1}(-\beta(h,h'))=-\beta(\hat{\phi}_{2}(h),\hat{\phi}_{2}(h')).$$ Notice that $\hat{\phi}_{1}$ is not in general injective, even if $\hat{\phi}_{\beta}$ is. \end{proof}
\begin{lem}\label{respectval} Let $G$ be an ordered abelian group with smallest convex subgroup $A$, and $L$ a mixed characteristic valued field with values in $A$ and a cross-section $\sigma$. Let $\overline{\phi}\in Aut(L)$, and $\hat{\phi}_{\beta}=(\hat{\phi}_{1},\hat{\phi}_{2})\in Aut(G_{\beta})$. If $\overline{\phi}\circ\sigma=\sigma\circ\hat{\phi}_{1}$, then $\phi:L(t^{H},\beta)\longrightarrow L(t^{H},\beta)$ such that $$\phi(\sum_{h\in S}a_{h}t^{h})=\sum_{h\in S}\overline{\phi}(a_{h})t^{\hat{\phi}(h)}$$ is an automorphism of $L(t^{H},\beta)$ as a valued field.\end{lem}
\begin{proof} Note that since $G\cong G_{\beta}$, any $\hat{\phi}\in Aut(G)$ uniquely determines an automorphism $\hat{\phi}_{\beta}\in Aut(G_{\beta})$. 
We now check that $\phi$ is an automorphism.
Clearly, $$\phi(\sum_{h\in S}a_{h}t^{h}+\sum_{h\in S'}b_{h}t^{h})=\phi(\sum_{h\in S}a_{h}t^{h})+\phi(\sum_{h\in S'}b_{h}t^{h})$$ as in the usual power series case.
For the multiplication, we have \begin{align*}&\phi(\sum_{h\in S}a_{h}t^{h}\cdot\sum_{h\in S'}b_{h}t^{h})=\\ =&\phi\left(\sum_{l\in S+S'}\left(\sum_{\substack{h\in S \\ h'\in S' \\ h+h'=l}}a_{h}b_{h'}\sigma(-\beta(h,h'))\right)t^{l}\right)=\\=&\sum_{l\in S+S'}\overline{\phi}\left(\sum_{\substack{h\in S \\ h'\in S' \\ h+h'=l}}a_{h}b_{h'}\sigma(-\beta(h,h'))\right)t^{\hat{\phi}_{2}(l)}= \\=&\sum_{l\in S+S'}\left(\sum_{\substack{h\in S \\ h'\in S' \\ h+h'=l}}\overline{\phi}(a_{h})\overline{\phi}(b_{h'})\overline{\phi}(\sigma(-\beta(h,h')))\right)t^{\hat{\phi}_{2}(l)}.\end{align*} On the other hand, we have \begin{align*}&\phi(\sum_{h\in S}a_{h}t^{h})\cdot\phi(\sum_{h\in S'}b_{h}t^{h})=\\ =&\sum_{h\in S}\overline{\phi}(a_{h})t^{\hat{\phi}_{2}(h)}\cdot\sum_{h\in S'}\overline{\phi}(b_{h})t^{\hat{\phi}_{2}(h)}=\\=&\sum_{l\in S+S'}\left(\sum_{\substack{h\in S \\ h'\in S' \\ h+h'=l}}\overline{\phi}(a_{h})\overline{\phi}(b_{h'})\sigma(-\beta(\hat{\phi}_{2}(h),\hat{\phi}_{2}(h')))\right)t^{\hat{\phi}_{2}(l)}.\end{align*}
\textbf{Claim.} For all $h,h'\in H$, $\overline{\phi}(\sigma(-\beta(h,h')))=\sigma(-\beta(\hat{\phi}_{2}(h),\hat{\phi}_{2}(h'))).$ 

By Lemma \ref{phioverbeta}, the claim becomes \begin{equation}\label{condition}\overline{\phi}(\sigma(-\beta(h,h')))=\sigma(\hat{\phi}_{1}(-\beta(h,h'))),\end{equation} i.e. the following diagram commutes 
\begin{center}
\begin{tikzpicture} 
  \matrix (m) [matrix of math nodes,row sep=3em,column sep=4em,minimum width=2em]
  {
     &A & A &\\
     & L & L &\\};
  \path[->]
    (m-1-2) edge node [above] {$\hat{\phi}_{1}$} (m-1-3)
            edge node [left] {$\sigma$} (m-2-2)
    (m-2-2) edge node [below] {$\overline{\phi}$} (m-2-3)
    (m-1-3) edge node [right] {$\sigma$} (m-2-3);
   
\end{tikzpicture}
\end{center}
which is true by hypothesis.
\end{proof}
Recall that in a discretely ordered abelian group the smallest convex subgroup $A$ is isomorphic to $\mathbb{Z}$, thus the unique surjective endomorphism of $A$ as an ordered group is the identity. Hence (\ref{condition}) forces $\overline{\phi}_{\restriction{\sigma(A)}}=id_{L}$. In particular, we know that in the case of a finitely ramified valued field with uniformizer $\pi$ and values in $\mathbb{Z}$, we have $\sigma(z)=\pi^{z}$, for all $z\in\mathbb{Z}$. Thus, (\ref{condition}) is equivalent to $\overline{\phi}(\pi^{z})=\pi^{z}$, for all $z\in\mathbb{Z}$.
\begin{prop} Let $K$ be an $\omega$-pseudo complete finitely ramified valued field of cardinality $\aleph_{1}$ with value group $G$ of cardinality $\aleph_{1}$, residue field $k$ and uniformizer $\pi$. Then if $\hat{\phi}$ is an automorphism of $G$, and $\overline{\phi}$ is an automorphism of $k$ such that its lift to $C(k)(\pi)$ fixes the $\pi^{th}$-powers, then there exists an automorphism of the valued field $K$ which lifts both $\hat{\phi}$ and $\overline{\phi}$ to $K$.\end{prop}
\begin{proof} As in the proof of Lemma \ref{respectval}, let $\hat{\phi}_{\beta}$ be the automorphism of $G_{\beta}$ induced by $\hat{\phi}$. Since $A\cong\mathbb{Z}$, we then have $\hat{\phi}_{\beta}=(id_{\mathbb{Z}},\hat{\phi}_{2})$, where $\hat{\phi}_{2}$ is an automorphism of $H$. By the hypothesis and the above observations on $\overline{\phi}$, we have that $\phi(\sum_{h\in H}a_{h}t^{h})=\sum_{h\in H}\overline{\phi}(a_{h})t^{{\hat{\phi}_{2}}}$ is an automorphism of $C(k)(\pi)(t^{H}, \beta)$, which is then isomorphic to $K$ as valued fields by Theorem \ref{mainthm}.  \end{proof}
\subsection{The equicharacteristic $p$ case}
In this section, we show that Theorem \ref{mainthm}, and in particular Corollary \ref{finextscharacterization}, can be generalized to the case in which the fixed residue field of characteristic $p$ is also equipped with a henselian valuation. Indeed, in the particular case in which this equicharacteristic $p$ valued field is isomorphic to a Hahn field, we have a whole description of the starting mixed characteristic valued field in terms of power series. As said in the Introduction, the question of characterizing valued fields in terms of power series in the equicharacteristic $p$ case has been addressed by many mathematicians. In \cite{Kaplansky42}, Kaplansky proves that an equicharacteristic $p$ henselian valued field K with no immediate extensions is uniquely determined by its value group and residue field if it satisfies what he calls the \textit{Hypothesis A}, that is (\cite{Kuhlmann22}) the value group is $p$-divisible and the residue field is perfect with no finite separable extensions of degree divisible by $p$. This follows from the uniqueness of the maximal immediate extension of a valued field which satisfies that hypothesis. This result works in the non-discrete case. The discrete case was addressed by Schilling, who proved the following result.
\begin{thm}{(\cite[Theorem 2]{Schilling37})} \label{schilling}Let $(K,v)$ be an equicharacteristic $p$ valued field with value group $G$ of finite rank $m$. If $(K,v)$ is pseudo complete, then it is isomorphic to the field of power series in $m$ variables and with coefficients in its residue field.\end{thm} 
Before stating the next result as a consequence of Theorem \ref{mainthm}, we need to recall the setting in which we are going to work. For this purpose, we use the standard decomposition of the valued field through the coarsening method. 

Let $(K,v)$ be a pseudo complete mixed characteristic valued field with finite ramification, residue field $k_{v}$, and valued in a discretely ordered abelian group $G$. Generally, the residue field can bring its own non-trivial valuation, that is an equicharacteristic $p$ valuation. In this case, the standard decomposition obtained by coarsening appears as follows: $$K\xrightarrow{\dot{v}}\mathring{K}\xrightarrow{v_{0}}k_{v}\xrightarrow{v_{p,p}}k_{p},$$ where \begin{itemize} \item $(K, \dot{v}, G/A, \mathring{K})$ is an equicharacteristic $0$ valued field; \item $(\mathring{K}, v_{0}, A, k_{v})$ is a mixed characteristic valued field; \item $(k_{v}, v_{p,p}, A, k_{p})$ is an equicharacteristic $p$ valued field;
\end{itemize}
where $A$ is the smallest convex subgroup of $G$.
\begin{cor} \label{cor2} Let $(K,v)$ be a pseudo complete mixed characteristic valued field of cardinality $\aleph_{1}$ with finite ramification, residue field $k_{v}$, and valued in a group $G$ of cardinality $\aleph_{1}$. Let $v_{p,p}$ be a henselian valuation on $k_{v}$ with respect to which it is complete, and let $k_{p}$ be the correspondent residue field. Then, $(K,v)$ is isomorphic to $L(t^{H},\beta)_{\aleph_{0}}$, where $L$ is a finite extension of the Cohen field $C(k_{p}((t)))$ over $k_{p}((t))$, $H=G/A$ and $\beta: H\times H\longrightarrow A$ a $2$-cocycle, modulo the Continuum Hypothesis.\end{cor}
\begin{proof} By the standard coarsening decomposition, the field $k_{v}$ has value group $A$, that is canonically isomorphic to $\mathbb{Z}$. Thus, by Schilling's theorem (Theorem \ref{schilling}), $(k_{v},v_{p,p})$ is isomorphic to the field of power series $k_{p}((t))$ on its residue field $k_{p}$. Then, the assert follows directly from Theorem \ref{mainthm}. \end{proof}
Now, we state this result in a saturated contest.
\begin{fact} \label{fact} An $\aleph_{1}$-saturated valued field is $\omega$-pseudo complete.\end{fact}
\begin{cor} \label{cor3} Let $(K,v)$ be a mixed characteristic valued field with finite ramification, value group $G$, and residue field $k_{v}$. Let $v_{p,p}$ be a henselian valuation on $k_{v}$, and let $k_{p}$ be the correspondent residue field. If $(K,v)$ is $\aleph_{1}$-saturated, then it is isomorphic to $L(t^{H}, \beta)_{\aleph_{0}}$, where $L$ is a finite extension of the Cohen field $C(k_{p}((t)))$ over $k_{p}((t))$, $H=G/A$ and $\beta:H\times H\longrightarrow A$ a $2$-cocycle, modulo the Continuum Hypothesis. \end{cor}
\begin{proof} If $(K,v)$ is $\aleph_{1}$-saturated, then also $k_{v}$ is. The henselian valuation $v_{p,p}$ on $k_{v}$ has values in $A$ that is an $\aleph_{1}$-saturated $\mathbb{Z}$-group. Then $A$ is a regular not divisible group. By \cite[Theorem 4]{Hong14}, $v_{p,p}$ is definable without parameters in the language of rings. This implies that $(k_{v},v_{p,p})$ is also $\aleph_{1}$-saturated as a valued field. Then by Fact \ref{fact}, $k_{v}$ is pseudo complete, and in particular, it is Cauchy complete. By \cite[Theorem 2]{Schilling37}, $(k_{v},v_{p,p})$ is isomorphic to $k_{p}((t))$ as valued fields. By Theorem \ref{mainthm} and Corollary \ref{cor2}, we have the assert. \end{proof}


\vspace{3mm}

\textbf{Acknowledgements.} I wish to thank my Ph.D. advisor Paola D'Aquino for her constant support and guidance during the preparation of my thesis, of which this research work is part. I would also like to thank Franziska Jahnke for pointing out some consequences of Theorem \ref{mainthm} which appear in Section \ref{cons}, and Pierre Touchard, for many helpful discussions. Special thanks go to Angus Macintyre, for suggesting this topic and for the interesting mathematical conversations. Finally, I would like to thank the Mathematical Science Research Institute (now Simons Laufer Mathematical Sciences Institute) for funding my staying in Berkeley during the ``Definability, decidability, and computability in number theory - Part 2" program, where this research began.
\bibliographystyle{acm}
\bibliography{TwistedPowerSeriesRef}

\begin{thebibliography}{10}

\bibitem{AJ19cohen}
{\sc Anscombe, S., and Jahnke, F.}
\newblock The model theory of {C}ohen rings.
\newblock {\em Confluentes Math. 14}, 2 (2022), 1--28.

\bibitem{AxK65i}
{\sc Ax, J., and Kochen, S.}
\newblock Diophantine problems over local fields. {I}.
\newblock {\em Amer. J. Math. 87\/} (1965), 605--630.

\bibitem{AxK65}
{\sc Ax, J., and Kochen, S.}
\newblock Diophantine problems over local fields ii. a complete set of axioms
  for p-adic number theory.*.
\newblock {\em American Journal of Mathematics 87\/} (1965), 631.

\bibitem{BvDDM88}
{\sc B\'{e}lair, L., van~den Dries, L., and Macintyre, A.}
\newblock Elementary equivalence and codimension in {$p$}-adic fields.
\newblock {\em Manuscripta Math. 62}, 2 (1988), 219--225.

\bibitem{CF67}
{\sc Cassels, J. W.~S., and Fr{\"o}hlich, A.}
\newblock {\em Algebraic number theory: proceedings of an instructional
  conference}.
\newblock Academic press, 1967.

\bibitem{CDLM13}
{\sc Cluckers, R., Derakhshan, J., Leenknegt, E., and Macintyre, A.}
\newblock Uniformly defining valuation rings in henselian valued fields with
  finite or pseudo-finite residue fields.
\newblock {\em Annals of Pure and Applied Logic 164}, 12 (2013), 1236--1246.
\newblock Logic Colloquium 2011.

\bibitem{cohen46}
{\sc Cohen, I.~S.}
\newblock On the structure and ideal theory of complete local rings.
\newblock {\em Transactions of the American mathematical Society 59}, 1 (1946),
  54--106.

\bibitem{ADM1}
{\sc {De Mase}, A.}
\newblock {Relative model completeness of henselian valued fields with finite
  ramification and various value groups}.
\newblock {\em arXiv e-prints\/} (Nov. 2023), arXiv:2311.00411 [math.LO].

\bibitem{EP05}
{\sc Engler, A.~J., and Prestel, A.}
\newblock {\em Valued fields}.
\newblock Springer Monographs in Mathematics. Springer-Verlag, Berlin, 2005.

\bibitem{HasseH24}
{\sc Hasse, H.}
\newblock Darstellbarkeit von {Z}ahlen durch quadratische {F}ormen in einem
  beliebigen algebraischen {Z}ahlk\"{o}rper.
\newblock {\em J. Reine Angew. Math. 153\/} (1924), 113--130.

\bibitem{Hong14}
{\sc Hong, J.}
\newblock Definable non-divisible henselian valuations.
\newblock {\em Bulletin of the London Mathematical Society 46\/} (02 2014).

\bibitem{Kaplansky42}
{\sc Kaplansky, I.}
\newblock {Maximal fields with valuations}.
\newblock {\em Duke Mathematical Journal 9}, 2 (1942), 303 -- 321.

\bibitem{Kochen74}
{\sc Kochen, S.}
\newblock The model theory of local fields.
\newblock In {\em ISILC Logic Conference: Proceedings of the International
  Summer Institute and Logic Colloquium, Kiel 1974\/} (2006), Springer,
  pp.~384--425.

\bibitem{Kuhlmann22}
{\sc Kuhlmann, F.-V.}
\newblock Subfields of algebraically maximal kaplansky fields.
\newblock {\em Communications in Algebra 50}, 8 (2022), 3346--3353.

\bibitem{20padic}
{\sc Macintyre, A.}
\newblock Twenty years of {$p$}-adic model theory.
\newblock In {\em Logic colloquium '84 ({M}anchester, 1984)}, vol.~120 of {\em
  Stud. Logic Found. Math.} North-Holland, Amsterdam, 1986, pp.~121--153.

\bibitem{Maclane38}
{\sc MacLane, S.}
\newblock The uniqueness of the power series representation of certain fields
  with valuations.
\newblock {\em Annals of Mathematics 39}, 2 (1938), 370--382.

\bibitem{Maclane39}
{\sc MacLane, S.}
\newblock Subfields and automorphism groups of p-adic fields.
\newblock {\em Annals of Mathematics 40}, 2 (1939), 423--442.

\bibitem{Mendelson61}
{\sc Mendelson, E.}
\newblock On non-standard models for number theory.
\newblock In {\em Essays on the foundations of mathematics}. Magnes Press,
  Hebrew Univ., Jerusalem, 1961, pp.~259--268.

\bibitem{Neumann49}
{\sc Neumann, B.~H.}
\newblock On ordered division rings.
\newblock {\em Transactions of the American Mathematical Society 66}, 1 (1949),
  202--252.

\bibitem{Poonen93}
{\sc Poonen, B.}
\newblock Maximally complete fields.
\newblock {\em Enseign. Math 39}, 1-2 (1993), 87--106.

\bibitem{PR84}
{\sc Prestel, A., and Roquette, P.}
\newblock {\em Formally {$p$}-adic fields}, vol.~1050 of {\em Lecture Notes in
  Mathematics}.
\newblock Springer-Verlag, Berlin, 1984.

\bibitem{Robinson96}
{\sc Robinson, D. J.~S.}
\newblock {\em A course in the theory of groups}, second~ed., vol.~80 of {\em
  Graduate Texts in Mathematics}.
\newblock Springer-Verlag, New York, 1996.

\bibitem{Schilling37}
{\sc Schilling, O. E.~G.}
\newblock Arithmetic in fields of formal power series in several variables.
\newblock {\em Annals of Mathematics 38}, 3 (1937), 551--576.

\bibitem{S79}
{\sc Serre, J.-P.}
\newblock {\em Local fields}, vol.~67 of {\em Graduate Texts in Mathematics}.
\newblock Springer-Verlag, New York-Berlin, 1979.
\newblock Translated from the French by Marvin Jay Greenberg.

\bibitem{Teich37}
{\sc Teichmüller, O.}
\newblock Diskret bewertete perfekte körper mit unvollkommenem
  restklassenkörper.
\newblock {\em Journal für die reine und angewandte Mathematik 176\/} (1937),
  141--152.

\end{thebibliography}
\end{document}